\documentclass[12pt,bezier]{article}
\usepackage{amssymb}
\usepackage{mathrsfs}
\usepackage{amsmath}
\usepackage{amsfonts,amsthm,amssymb}
\usepackage{amsfonts}
\usepackage{graphics}
\usepackage{epstopdf}
\usepackage[]{caption2}
\usepackage{color}
\newtheorem{lem}{Lemma}
\newtheorem{thm}{Theorem}

\begin{document}
\title{Contractible Subgraphs of Contraction Critically Quasi $5$-Connected Graphs}

\author{Shuai Kou$^a$, \quad Chengfu Qin$^b$, \quad Weihua Yang$^a$\footnote{Corresponding author. E-mail: ywh222@163.com,~yangweihua@tyut.edu.cn}\\
\small $^a$Department of Mathematics, Taiyuan University of Technology, Taiyuan 030024, China\\
\small $^b$Department of Mathematics, Guangxi Teachers Education University, Nanning 530001, China\\
}
\date{}

\maketitle {\flushleft\bf Abstract:} {\small  Let $G$ be a contraction critically quasi $5$-connected graph on at least $14$ vertices. If there is a vertex $x\in V_{4}(G)$ such that $G[N_{G}(x)]\cong K_{1,3}$ or $G[N_{G}(x)]\cong C_{4}$, then $G$ has a quasi $5$-contractible subgraph $H$ such that $0<\|V(H)\|<4$.}


\section{Introduction}
In this paper, we only consider finite simple undirected graphs, with undefined terms and notations following \cite{Bondy}.
For a graph $G$, let $V(G)$ and $E(G)$ denote the set of vertices of $G$ and the set of edges of $G$, respectively. Let $V_{k}(G)$ denote the set of vertices of degree $k$ in $G$.

A $cut$ of a connected graph $G$ is a subset $V^{\prime}(G)$ of $V(G)$ such that $G-V^{\prime}(G)$ is disconnected. A $k$-$cut$ is a cut of $k$ elements. Suppose $T$ is a $k$-cut of $G$. We say that $T$ is a $nontrivial$ $k$-$cut$, if the components of $G-T$ can be partitioned into subgraphs $G_{1}$ and $G_{2}$ such that $|V(G_{1})|\geq 2$ and $|V(G_{2})|\geq 2$.
By $\mathscr{T}(G):=\{T\subseteq V(G): T$ is a cut of $G, |T|=\kappa(G)\}$, we denote the set of $smallest$ $cuts$ of $G$. For $T\in\mathscr{T}(G)$, the union of the vertex sets of at least one but not of all components of $G-T$ is called a $T$-$fragment$ of $G$ or, briefly, a $fragment$.
A ($k$-$1$)-connected graph is $quasi$ $k$-$connected$ if it has no nontrivial ($k$-$1$)-cuts. Clearly, every $k$-connected graph is quasi $k$-connected.

An edge $e=xy$ of $G$ is said to be $contracted$ if it is deleted and its ends are identified, the resulting graph is denoted by $G/e$. And the new vertex in $G/e$ is denoted by $\overline{xy}$.
Let $k$ be an integer such that $k\geq 2$ and let $G$ be a $k$-connected graph with $|V(G) |\geq k+2$. An edge $e$ of $G$ is said to be $k$-$contractible$ if the contraction of the edge results in a $k$-connected graph. Note that, in the contraction, we replace each resulting pair of double edges by a simple edge.
A subgraph of a $k$-connected graph $G$ is said to be $k$-contractible if its contraction (i.e., identifying each components to a single vertex, removing each of the resulting loops and, finally, replacing each of the resulting duplicate edges by a single edge) results again a $k$-connected graph.

By Tutte's constructive characterization of $3$-connected graphs\cite{Tutte}, we see that every $3$-connected graph except $K_{4}$ has a $3$-contractible edge. Thomassen\cite{Thomassen} showed that for $k\geq4$, there are infinitely many $k$-connected $k$-regular graphs which do not have a $k$-contractible edge. These graphs are said to be $contraction$ $critically$ $k$-$connected$.

However, every $4$-connected graph on at least seven vertices can be reduced to a smaller $4$-connected graph by contracting one or two edges subsequently. So, naturally, the question that whether there exist positive integers $b$ and $h$ such that every $k$-connected graph on more than $b$ vertices can be reduced to a more smaller $k$-connected graph by contracting less than $h$ edges for every $k\geq1$ is posted\cite{Kriesell1}. It is holds for $k=1, 2, 3$ and $4$. But for $k\geq6$, such a statement fails since toroidal triangulations of large face width is a counterexample\cite{Kriesell2}. The question is still open for $k=5$.

We focus on quasi $5$-connected graphs and obtain the following results.

\begin{thm}\label{thm1}
Let $G$ be a contraction critically quasi $5$-connected graph on at least $14$ vertices. If there is a vertex $x\in V_{4}(G)$ such that $G[N_{G}(x)]\cong K_{1,3}$, then $G$ has a quasi $5$-contractible subgraph $H$ such that $0<\|V(H)\|<4$.
\end{thm}

\begin{thm}\label{thm2}
Let $G$ be a contraction critically quasi $5$-connected graph on at least $14$ vertices. If there is a vertex $x\in V_{4}(G)$ such that $G[N_{G}(x)]\cong C_{4}$, then $G$ has a quasi $5$-contractible subgraph $H$ such that $0<\|V(H)\|<4$.
\end{thm}

\section{Preliminaries}

\begin{lem}[\cite{Dong}]\label{lem1}
Let $G$ be a $5$-connected graph, then there exists an edge $e\in E(G)$ such that $G/e$ is quasi $5$-connected.
\end{lem}

\begin{lem}\label{lem2}
Let $G$ be a contraction critically quasi $5$-connected graph, $x\in V(G)$. If $d(x)=4$, then $G[N(x)]$ contains no $K_{3}$-subgraph.
\end{lem}

\begin{proof}
By contradiction. Suppose that $N(x)=\{x_{1}, x_{2}, x_{3}, x_{4}\}$ and $G[\{x_{1}, x_{2}, x_{3}\}]\cong K_{3}$. Since $G$ is contraction critically, $G/xx_{4}$ is not quasi $5$-connected and then $\kappa (G/xx_{4})\leq4$ by Lemma \ref{lem1}.
If $\kappa (G/xx_{4})<4$, then $G/xx_{4}$ has a $3$-cut $T^{\prime}$ such that $\overline{xx_{4}}\in T^{\prime}$. Since $G[\{x_{1}, x_{2}, x_{3}\}]\cong K_{3}$, we have $T^{\prime}-\{\overline{xx_{4}}\}+\{x_{4}\}$ forms a $3$-cut of $G$, contradicts with $G$ is quasi $5$-connected.
If $\kappa (G/xx_{4})=4$, then $G/xx_{4}$ has a nontrivial $4$-cut $T^{\prime}$ such that $\overline{xx_{4}}\in T^{\prime}$. Thus $T^{\prime}-\{\overline{xx_{4}}\}+\{x_{4}\}$ is a nontrivial $4$-cut of $G$, a contradiction.
\end{proof}

Let $G$ be a contraction critically quasi $5$-connected graph. By Lemma \ref{lem1}, we know that $\kappa(G)=4$. Furthermore, we have $\delta(G)=4$, otherwise, every $4$-cut of $G$ is nontrivial, contradicting that $G$ is quasi $5$-connected.
By Lemma \ref{lem2}, neighbour sets of all $4$-degree vertices have at most $7$ types in $G$ (see Figure \ref{fig1}). In this paper, we consider that there exists a $4$-degree vertex $x$ has type $6$ or type $7$ in $G$. In other words, we consider that there exists a vertex $x\in V_{4}(G)$ such that $G[N_{G}(x)]\cong K_{1,3}$ or $G[N_{G}(x)]\cong C_{4}$.

\begin{figure}
  \centering
  \includegraphics{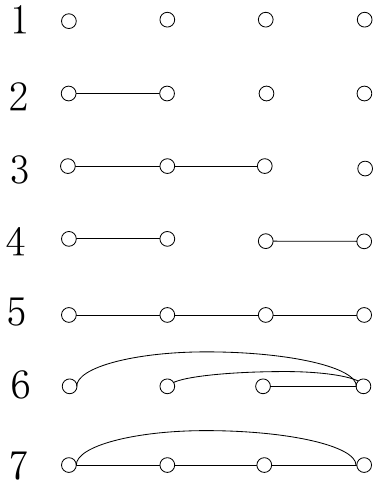}\\
  \caption{The seven neighborhood types}
  \label{fig1}
\end{figure}

\begin{lem}\label{lem3}
Let $G$ be a quasi $5$-connected graph and $xy\in E(G)$. If $\delta(G/xy)\geq4$, then $G/xy$ is $4$-connected.
\end{lem}
\begin{proof}
Assume, to the contrary, that $G/xy$ is not $4$-connected, then there exists a $3$-cut $T^{\prime}$ of $G/xy$ such that $\overline{xy}\in T^{\prime}$. Since $\delta(G/xy)\geq4$, each component of $G-T^{\prime}$ has at least $2$ vertices. Hence, $T=T^{\prime}-\overline{xy}+\{x, y\}$ is a $4$-cut of $G$. Since $G$ is quasi $5$-connected, $T$ is nontrivial, a contradiction.
\end{proof}

Let us close this section with some conventions.
In the following, if the graph $G^{\prime}$ obtained from $G$ by contracting several edges $A$, $B$, $C$ to vertices $A^{\prime}$, $B^{\prime}$, $C^{\prime}$, respectively. Then let $F$, $T$, $\overline{F}$ be the sets in $G$ corresponding to $F^{\prime}$, $T^{\prime}$, $\overline{F^{\prime}}$ in $G^{\prime}$. That is, in each of these sets, we replace the vertices $A^{\prime}$, $B^{\prime}$, $C^{\prime}$ by the vertices in the sets $A$, $B$, $C$, respectively. In addition, we agree that in all the figures below, solid lines represent edges that must exist, and dotted lines represent edges that may exist. A solid vertex means that the degree of this vertex is full.

\section{Proof of Theorem \ref{thm1}}
Let $G$ be a quasi $5$-connected graph on at least $14$ vertices. We suppose that there exists a vertex $x\in V_{4}(G)$ and $N_{G}(x)=\{x_{1}, x_{2}, x_{3}, x_{4}\}$.
In Lemma \ref{lem4}, \ref{lem5}, \ref{lem6}, \ref{lem7} and \ref{lem8}, we consider that $G[N_{G}(x)]\cong K_{1,3}$ and $\{x_{1}x_{4}, x_{2}x_{4}, x_{3}x_{4}\}\subseteq E(G)$, then we have $d_{G}(x_{4})\geq5$, for otherwise, $\{x_{1}, x_{2}, x_{3}\}$ forms a $3$-cut of $G$, a contradiction.

\begin{figure}[t]
  \centering
  \includegraphics{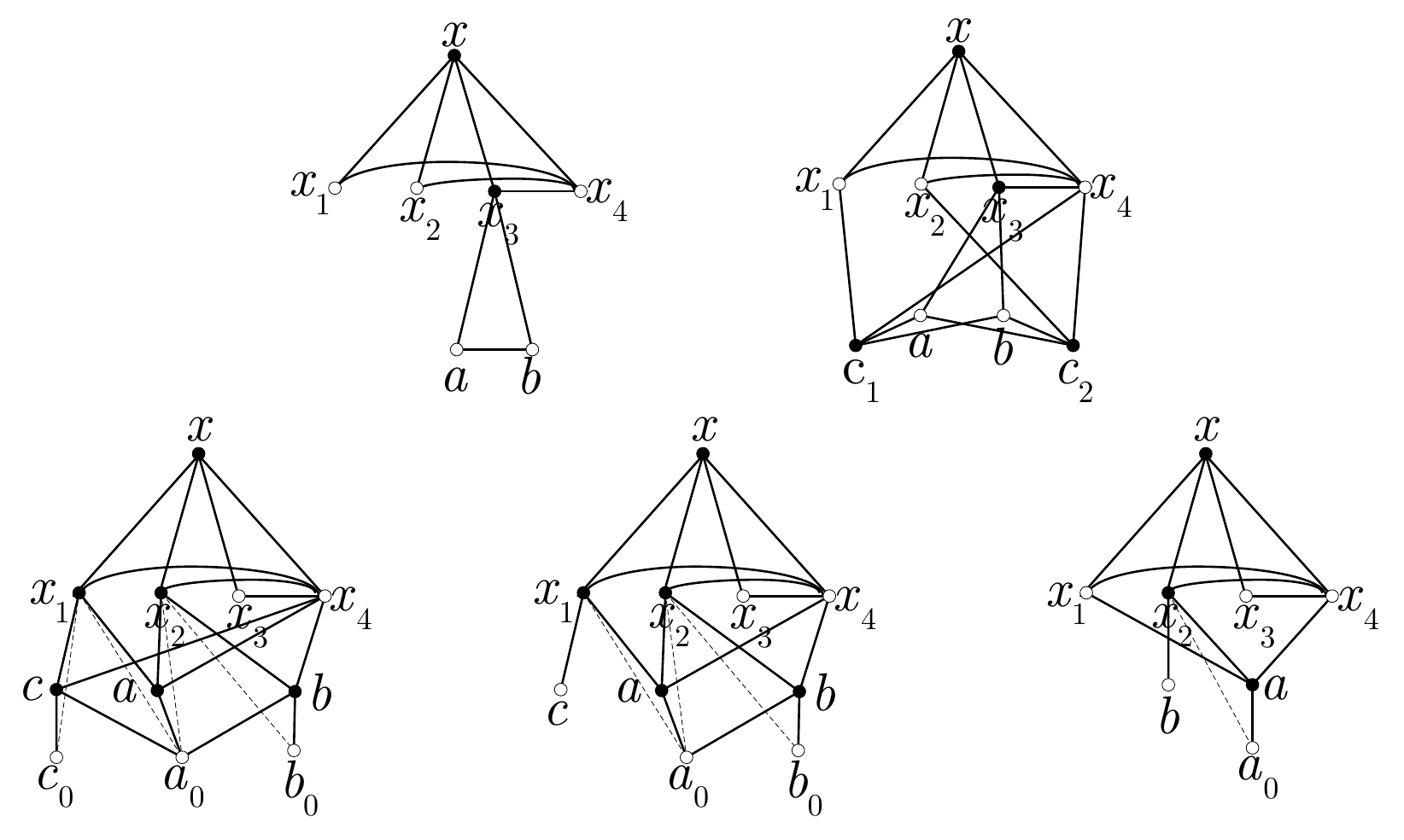}\\
  \caption{The structures of Lemma \ref{lem4}-Lemma \ref{lem8}}\label{fig2}
\end{figure}

\begin{lem}\label{lem4}
Let $G$ be a quasi $5$-connected graph on at least $14$ vertices. If $G$ has the first structure in Figure \ref{fig2}, then $xx_{3}$ is quasi $5$-contractible.
\end{lem}

\begin{proof}
Clearly, $\delta(G/xx_{3})=4$. By Lemma \ref{lem3}, $G/xx_{3}$ is $4$-connected. Suppose, to the contrary, that $xx_{3}$ is not quasi $5$-contractible, then $G/xx_{3}$ has a nontrivial $4$-cut $T^{\prime}$ and $\overline{xx_{3}}\in T^{\prime}$. Then $T=T^{\prime}-\{\overline{xx_{3}}\}+\{x, x_{3}\}$ is a $5$-cut of $G$. If $x_{4}\notin T$, then $T-\{x\}$ is a nontrivial $4$-cut of $G$, a contradiction. Thus, $x_{4}\in T$. Since $N_{G}(x_{3})=\{x, x_{4}, a, b\}$ and $ab\in E(G)$, $T-\{x_{3}\}$ is a nontrivial $4$-cut of $G$, a contradiction.
\end{proof}

\begin{lem}\label{lem5}
Let $G$ be a quasi $5$-connected graph on at least $14$ vertices. If $G$ has the second structure in Figure \ref{fig2}, then $xx_{3}$ is quasi $5$-contractible.
\end{lem}
\begin{proof}
If $ab\in E(G)$, $G$ has the first structure in Figure \ref{fig2}. By Lemma \ref{lem4}, the edge $xx_{3}$ is quasi $5$-contractible. So we suppose that $ab\notin E(G)$. By Lemma \ref{lem3}, $G/xx_{3}$ is $4$-connected. If $G/xx_{3}$ is not quasi $5$-connected, then
$G/xx_{3}$ has a nontrivial $4$-cut $T^{\prime}$ and $\overline{xx_{3}}\in T^{\prime}$. Let $F^{\prime}$ be a $T^{\prime}$-fragment, $\overline{F^{\prime}}=G/xx_{3}-T^{\prime}-F^{\prime}$. Thus, $T=T^{\prime}-\{\overline{xx_{3}}\}+\{x, x_{3}\}$ and $x_{4}\in T$. It follows that $x_{1}$ and $x_{2}$ are one in $F$ and the other in $\overline{F}$, and so are $a$ and $b$. Without loss of generality, we can assume that $\{x_{1}, a\}\subset F$ and $\{x_{2}, b\}\subset\overline{F}$. Thus, $T=\{x, x_{3}, x_{4}, c_{1}, c_{2}\}$, and thus $N_{G}(\{x, x_{3}, c_{1}, c_{2}\})\cap F=\{x_{1}, a\}$ and $N_{G}(\{x, x_{3}, c_{1}, c_{2}\})\cap\overline{F}=\{x_{2}, b\}$. Since $|V(G)|\geq14$, $|F|\geq3$ or $|\overline{F}|\geq3$. So $\{x_{4}, x_{1}, a\}$ or $\{x_{4}, x_{2}, b\}$ are $3$-cut of $G$, which is absurd.
\end{proof}

By Lemma \ref{lem4} and Lemma \ref{lem5}, $G$ has no the first and the second structure in Figure \ref{fig2} if $G$ is a contraction critically quasi $5$-connected graph on at least $14$ vertices.

\begin{lem}\label{lem6}
Let $G$ be a quasi $5$-connected graph on at least $14$ vertices. If $G$ has the third structure in Figure \ref{fig2}, then the graph $G^{\prime}$ obtained from $G$  by contracting $A:=\{a, a_{0}\}$, $B:=\{b, b_{0}\}$, $C:=\{c, c_{0}\}$ to vertices $A^{\prime}$, $B^{\prime}$, $C^{\prime}$, respectively, is still quasi $5$-connected.
\end{lem}
\begin{proof}
If $N_{G}(a_{0})=\{a, b, c, b_{0}\}$, then $\{x_{3}, x_{4}, b_{0}, c_{0}\}$ forms a nontrivial $4$-cut of $G$ since $|V(G)|\geq14$. So $N_{G}(a_{0})\neq\{a, b, c, b_{0}\}$. Similarly, $N_{G}(a_{0})\neq\{a, b, c, c_{0}\}$. Therefore, $\delta(G^{\prime})=4$.

Firstly, we show that $G^{\prime}$ is $4$-connected. Assume, to the contrary, that $G^{\prime}$ is not $4$-connected, then $G^{\prime}$ has a cut $T^{\prime}$ such that $|T^{\prime}|\leq3$.
Clearly, $|T^{\prime}|\leq2$ is impossible. So $|T^{\prime}|=3$, and, $|T^{\prime}\cap\{A^{\prime}, B^{\prime}, C^{\prime}\}|\geq2$.
However, $T^{\prime}\cap\{A^{\prime}, B^{\prime}, C^{\prime}\}\neq\{A^{\prime}, B^{\prime}\}$, for otherwise, $T-\{b\}$ is a nontrivial $4$-cut of $G$.
Similarly, we have $T^{\prime}\cap\{A^{\prime}, B^{\prime}, C^{\prime}\}\neq\{A^{\prime}, C^{\prime}\}$. Therefore, the set $T^{\prime}\cap\{A^{\prime}, B^{\prime}, C^{\prime}\}$ is  $\{B^{\prime}, C^{\prime}\}$ or $\{A^{\prime}, B^{\prime}, C^{\prime}\}$. In the former, $T-\{b\}$ is a nontrivial $4$-cut of $G$ or $\{a_{0}, b_{0}, c_{0}\}$ forms a $3$-cut of $G$. In the latter, $\{a_{0}, b_{0}, c_{0}\}$ forms a $3$-cut of $G$. Both contradict that $G$ is a quasi $5$-connected graph. This proves that $G^{\prime}$ is $4$-connected.

Now we show that $G^{\prime}$ is quasi $5$-connected. Otherwise, there exists a nontrivial $4$-cut $T^{\prime}$ of $G^{\prime}$ and $|T^{\prime}\cap\{A^{\prime}, B^{\prime}, C^{\prime}\}|\geq1$. Let $F^{\prime}$ be a $T^{\prime}$-fragment and $\overline{F^{\prime}}=G^{\prime}-T^{\prime}-F^{\prime}$.

{\bf Claim 1.} $\{A^{\prime}, B^{\prime}, C^{\prime}\}\nsubseteq T^{\prime}$.

By contradiction, then $|T|=7$ and $T\supset\{a, a_{0}, b, b_{0}, c, c_{0}\}$. Whether $x\in T$ or $x\notin T$, we both have $N_{G}(\{a, b, c\})\cap F=\emptyset$ or $N_{G}(\{a, b, c\})\cap \overline{F}=\emptyset$, and, thus, $T-\{a, b, c\}$ forms a nontrivial $4$-cut of $G$, a contradiction.
This proves Claim $1$.

{\bf Claim 2.} $\{A^{\prime}, B^{\prime}\}\nsubseteq T^{\prime}$.

For otherwise, $|T|=6$ and $T\supset\{a, a_{0}, b, b_{0}\}$. Without loss of generality, we assume that $\{c, c_{0}\}\subseteq F$. Since $cx_{1}\in E(G)$ and $cx_{4}\in E(G)$, $\{x_{1}, x_{4}\}\subset F\cup T$. If $x_{2}\in F\cup T$, then $N_{G}(\{a, b\})\subset F\cup T$, and, thus, $T-\{a, b\}$ forms a nontrivial $4$-cut of $G$. So $x_{2}\in\overline{F}$, and thus $x_{4}\in T$ and $x_{1}\in F$, and then $T=\{a, a_{0}, b, b_{0}, x_{4}, x\}$.
If $x_{3}\in\overline{F}$, then $N_{G}(\{a, b, x\})\cap F=\{x_{1}\}$. Hence $T-\{a, b, x\}+\{x_{1}\}$ is a nontrivial $4$-cut of $G$, a contradiction. If $x_{3}\in F$, then $N_{G}(x_{2})\subset T$, and then $\{a_{0}, b_{0}, x_{4}\}$ is a $3$-cut of $G$, a contradiction. This proves Claim $2$.

Similarly, we have $\{A^{\prime}, C^{\prime}\}\nsubseteq T^{\prime}$.

{\bf Claim 3.} $\{B^{\prime}, C^{\prime}\}\nsubseteq T^{\prime}$.

By contradiction. Without loss of generality, we assume that $\{a, a_{0}\}\subset F$, then $\{x_{1}, x_{2}, x_{4}\}\subset F\cup T$. Hence $N_{G}(\{b, c\})\subset F\cup T$, implying that $T-\{b, c\}$ is a nontrivial $4$-cut of $G$. This proves Claim $3$.

By Claim $1$, $2$ and $3$, we have $|T^{\prime}\cap\{A^{\prime}, B^{\prime}, C^{\prime}\}|=1$.
If $B^{\prime}\in T^{\prime}$, then $\{a, a_{0}, c, c_{0}\}\subset F$ without loss of generality. Hence $N_{G}(b)\subset F\cup T$, and then $T-\{b\}$ is a nontrivial $4$-cut of $G$, a contradiction. So $B^{\prime}\notin T^{\prime}$.
Similarly, we have $C^{\prime}\notin T^{\prime}$.
Therefore, $T^{\prime}\cap\{A^{\prime}, B^{\prime}, C^{\prime}\}=\{A^{\prime}\}$.
If $\{B^{\prime}, C^{\prime}\}\subseteq F^{\prime}$ or $\{B^{\prime}, C^{\prime}\}\subseteq \overline{F^{\prime}}$, then $T-\{a\}$ is a nontrivial $4$-cut of $G$. So we can assume that $B^{\prime}\in F^{\prime}$ and $C^{\prime}\in\overline{F^{\prime}}$. In other words, $\{a, a_{0}\}\in T$, $\{b, b_{0}\}\in F$ and $\{c, c_{0}\}\in\overline{F}$. Hence $|F|\geq3$ and $|\overline{F}|\geq3$. Since $bx_{4}\in E(G)$ and $cx_{4}\in E(G)$, $x_{4}\in T$. Since $bx_{2}\in E(G)$ and $cx_{1}\in E(G)$, $x_{2}\in F$ and $x_{1}\in\overline{F}$ (for otherwise, $T-\{a\}$ is a nontrivial $4$-cut of $G$ ). And then $x\in T$, implying that $T-\{a, x\}+\{x_{1}\}$ or $T-\{a, x\}+\{x_{2}\}$ forms a nontrivial $4$-cut of $G$, a contradiction.
\end{proof}

\begin{lem}\label{lem7}
Let $G$ be a contraction critically quasi $5$-connected graph on at least $14$ vertices. If $G$ has the forth structure in Figure \ref{fig2}, then $G$ has a quasi $5$-contractible subgraph $H$ such that $\|H\|\leq3$.
\end{lem}
\begin{proof}
Let $G^{\prime}$ be the graph obtained from $G$ by contracting $A:=\{a, a_{0}\}$, $B:=\{b, b_{0}\}$ to vertices $A^{\prime}$, $B^{\prime}$, respectively. We claim that $\delta(G^{\prime})=4$.
Suppose that $\delta(G^{\prime})\neq4$, then $d_{G}(a_{0})=4$ and $b_{0}\in N_{G}(a_{0})$. If $N_{G}(a_{0})=\{a, b, b_{0}, x_{2}\}$, then $\{x_{3}, x_{4}, c, b_{0}\}$ is a nontrivial $4$-cut of $G$ since $|V(G)|\geq14$, a contradiction. So $d_{G}(x_{2})=4$, and then $G$ has the first structure in Figure \ref{fig2}. By Lemma \ref{lem4}, $x_{2}b$ is quasi $5$-contractible, contradicts that $G$ is a contraction critically quasi $5$-connected graph. Therefore, $\delta(G^{\prime})=4$.

Similar to Lemma \ref{lem6}, we know that $G^{\prime}$ is $4$-connected. If $G^{\prime}$ is quasi $5$-connected, we obtain this result immediately. So we assume that $G^{\prime}$ is not quasi $5$-connected, then $G^{\prime}$ has a nontrivial $4$-cut $T^{\prime}$. Let $F^{\prime}$ be a $T^{\prime}$-fragment and $\overline{F^{\prime}}=G^{\prime}-T^{\prime}-F^{\prime}$.
We claim that $|T^{\prime}\cap\{A^{\prime}, B^{\prime}\}|=1$. If $\{A^{\prime}, B^{\prime}\}\subset T^{\prime}$, then $x_{1}$ and $x_{2}$ are one in $F$ and the other in $\overline{F}$. Now $T=\{a, a_{0}, b, b_{0}, x_{4}, x\}$, implying that $N_{G}(x_{2})\subseteq T$. If $x_{3}\in F$, then $\{a_{0}, b_{0}, x_{4}\}$ forms a $3$-cut of $G$, a contradiction. Hence, $x_{3}\in\overline{F}$, and thus, $T_{0}=T-\{a, b, x\}+\{x_{1}\}$ is a $4$-cut of $G$. Consequently, $F=\{x_{1}, c\}$ and $N_{G}(c)=\{x_{1}, a_{0}, b_{0}, x_{4}\}$, implying that $\{x_{3}, x_{4}, a_{0}, b_{0}\}$ forms a nontrivial $4$-cut of $G$, a contradiction. So $|T^{\prime}\cap\{A^{\prime}, B^{\prime}\}|=1$.

If $B^{\prime}\in T^{\prime}$, then $T-\{b\}$ is a nontrivial $4$-cut of $G$, a contradiction. So $A^{\prime}\in T^{\prime}$. Without loss of generality, we assume that $B^{\prime}\in F^{\prime}$. Note that $bx_{2}\in E(G)$, then $x_{2}\in F$ and $x_{1}\in\overline{F}$. Thus, $T\supset\{a, a_{0}, x_{4}, x\}$ and $|F|\geq3$, then $x_{3}\in F$. From that we have $\overline{F}=\{x_{1}, c\}$, and thus the vertex $c$ is a $4$-degree vertex such that $\{x_{1}, a_{0}, x_{4}\}\subset N_{G}(c)$. This implies that $G$ has the third structure in Figure \ref{fig2}. By Lemma \ref{lem6}, we can obtain this result.
\end{proof}

\begin{lem}\label{lem8}
Let $G$ be a contraction critically quasi $5$-connected graph on at least $14$ vertices. If $G$ has the fifth structure in Figure \ref{fig2}, then there exists a quasi $5$-contractible subgraph $H$ such that $\|H\|\leq3$.
\end{lem}

\begin{proof}
Firstly, we claim that $\delta(G/aa_{0})=4$. For otherwise, $d_{G}(x_{1})=\{x, x_{4}, a, a_{0}\}$, then $\{x_{3}, x_{4}, b, a_{0}\}$ is a nontrivial $4$-cut of $G$, a contradiction. So $\delta(G/aa_{0})=4$. Thus, $G/aa_{0}$ is $4$-connected by Lemma \ref{lem3}. Since $G$ is contraction critically, $G/aa_{0}$ is not quasi $5$-connected.
So there exists a nontrivial $4$-cut $T^{\prime}$ of $G/aa_{0}$. Let $F^{\prime}$ be a $T^{\prime}$-fragment and $\overline{F^{\prime}}=G/aa_{0}-T^{\prime}-F^{\prime}$. Without loss of generality, we can assume that $x_{1}\in F$ and $x_{2}\in\overline{F}$. Since $|V(G)|\geq14$, we have $x_{3}\notin T$.

{\bf Case 1.} $x_{3}\in\overline{F}$.

Then $|F|=2$. We assume that $F=\{x_{1}, c\}$, then $d_{G}(c)=4$ and $\{x_{1}, a_{0}, x_{4}\}\subset N_{G}(c)$. Let $N_{G}(c)=\{x_{1}, a_{0}, x_{4}, c_{0}\}$, then we have $N_{G}(x_{1})\subseteq \{a, x, c, x_{4}, a_{0}, c_{0}\}$.
If $\{c_{0}\}=\{b\}$, then $\{a_{0}, x_{4}, b, x_{3}\}$ forms a nontrivial $4$-cut of $G$. So the $10$ vertices in the list $L=\{x, x_{1}, x_{2}, x_{3}, x_{4}, a, a_{0}, b, c, c_{0}\}$ are pairwise distinct. Hence, there exists the forth structure in Figure \ref{fig2}.

{\bf Case 2.} $x_{3}\in F$.

Then $\overline{F}=\{x_{2}, b\}$, and thus $d_{G}(b)=4$ and $\{x_{2}, a_{0}, x_{4}\}\subset N_{G}(b)$. Let $N_{G}(b)=\{x_{2}, a_{0}, x_{4}, b_{0}\}$. Now, we focus on the graph $G_{0}^{\prime}$ obtained from $G$ by contracting $A:=\{a, a_{0}\}$, $B:=\{x, x_{2}\}$ to vertices $A^{\prime}$, $B^{\prime}$ respectively. Clearly, $G_{0}^{\prime}$ is $4$-connected.
If $G_{0}^{\prime}$ is quasi $5$-connected, we get the Lemma immediately. So we suppose that $G_{0}^{\prime}$ is not quasi $5$-connected, then $G_{0}^{\prime}$ has a nontrivial $4$-cut $T_{0}^{\prime}$. Let $F_{0}^{\prime}$ be a $T_{0}^{\prime}$-fragment and $\overline{F_{0}^{\prime}}=G_{0}^{\prime}-T_{0}^{\prime}-F_{0}^{\prime}$. let $F_{0}$, $T_{0}$, $\overline{F_{0}}$ be the sets in $G$ corresponding to $F_{0}^{\prime}$, $T_{0}^{\prime}$, $\overline{F_{0}^{\prime}}$ in $G_{0}^{\prime}$.

If $|T_{0}^{\prime}\cap\{A^{\prime}, B^{\prime}\}|=1$, then we can easily find a nontrivial $4$-cut of $G$. So $\{A^{\prime}, B^{\prime}\}\subseteq T_{0}^{\prime}$. Then $|T_{0}|=6$ and $T_{0}\supset\{a, a_{0}, x, x_{2}\}$. Thus $x_{4}\in T_{0}$, and thus $x_{1}\notin T_{0}$. Without loss of generality, we suppose that $x_{1}\in F_{0}$, then $b\in\overline{F_{0}}$, and thus $|F_{0}|=2$. Let $F_{0}=\{x_{1}, c\}$, then $d_{G}(c)=4$ and $\{x_{1}, x_{4}, a_{0}\}\subseteq N_{G}(c)$. So there exists the third structure in Figure \ref{fig2}.

By Case $1$ and Case $2$, we can derive this lemma directly.
\end{proof}

Now we are prepared to prove our Theorem \ref{thm1}.

{\bf Proof of Theorem 1.} Let $N_{G}(x)=\{x_{1}, x_{2}, x_{3}, x_{4}\}$. Without loss of generality, we suppose that $x_{i}x_{4}\in E(G)$ $(i=1, 2, 3)$. So, $d(x_{4})\geq5$, and then $G/xx_{j}$ ($j=1,2,3$) is $4$-connected by Lemma \ref{lem3}. Since $G$ is a contraction critically quasi $5$-connected graph, $G/xx_{j}$ ($j=1,2,3$) is not quasi $5$-connected. So $G/xx_{1}$ and $G/xx_{2}$ have nontrivial $4$-cuts $T_{1}^{\prime}$ and $T_{2}^{\prime}$, separately. And we have $\overline{xx_{1}}\in T_{1}^{\prime}$, $\overline{xx_{2}}\in T_{2}^{\prime}$. For $i\in\{1, 2\}$, let $F_{i}^{\prime}$ be a fragment of $G/xx_{i}-T_{i}^{\prime}$, $\overline{F_{i}^{\prime}}=G/xx_{i}-T_{i}^{\prime}-F_{i}^{\prime}$. Let $F_{i}$, $T_{i}$, $\overline{F_{i}}$ be the sets in $G$ corresponding to $F_{i}^{\prime}$, $T_{i}^{\prime}$, $\overline{F_{i}^{\prime}}$ in $G/xx_{i}$. Then $T_{i}\supset\{x, x_{i}\}$, $|F_{i}|\geq2$, $|\overline{F_{i}}|\geq2$, and $x_{4}\in T_{1}\cap T_{2}$. Without loss of generality, we assume that $x_{1}\in T_{1}\cap F_{2}$, $x_{2}\in F_{1}\cap T_{2}$ and $x_{3}\in \overline{F_{1}}\cap\overline{F_{2}}$.
Let $X_{1}=(T_{1}\cap F_{2})\cup(T_{1}\cap T_{2})\cup(F_{1}\cap T_{2})$, $X_{2}=(T_{1}\cap F_{2})\cup(T_{1}\cap T_{2})\cup(\overline{F_{1}}\cap T_{2})$, $X_{3}=(\overline{F_{1}}\cap T_{2})\cup(T_{1}\cap T_{2})\cup( T_{1}\cap\overline{F_{2}})$ and $X_{4}=(F_{1}\cap T_{2})\cup(T_{1}\cap T_{2})\cup(T_{1}\cap\overline{F_{2}})$.

{\bf Case 1.} $|X_{2}|\neq5$.

If $|X_{2}|<5$, then $|X_{4}|>5$ by the fact that $|X_{2}|+|X_{4}|=10$. Thus $F_{1}\cap\overline{F_{2}}=\emptyset$ and $|\overline{F_{2}}\cap T_{1}|\leq1$. If $\overline{F_{2}}\cap T_{1}=\emptyset$, then $T_{2}-\{x_{2}\}$ forms a nontrivial $4$-cut of $G$, a contradiction. So $|\overline{F_{2}}\cap T_{1}|=1$. Hence we have $|X_{1}|=5$. Since $|F_{1}|\geq2$, $|F_{1}\cap F_{2}|=1$. Then $G$ has the fifth structure in Figure \ref{fig2}.

If $|X_{2}|>5$, by using the same argument as for the $|X_{2}|<5$, we can deduce that $G$ has the fifth structure in Figure \ref{fig2}.

{\bf Case 2.} $|X_{2}|=5$.

We claim that $|\overline{F_{2}}\cap T_{1}|=1$. If $|\overline{F_{2}}\cap T_{1}|=2$, we have $|X_{1}|=4$. Then $F_{1}\cap F_{2}=\emptyset$, and then $T_{2}-\{x_{2}\}$ forms a nontrivial $4$-cut of $G$, a contradiction.
If $|\overline{F_{2}}\cap T_{1}|=0$, then $|\overline{F_{1}}\cap T_{2}|=0$. Since $\overline{F_{1}}\cap\overline{F_{2}}\neq\emptyset$, $|X_{3}|\geq4$. Then $|T_{1}\cap T_{2}|=4$, and then $|F_{2}\cap T_{1}|=|F_{1}\cap T_{2}|=|\overline{F_{1}}\cap F_{2}|=|\overline{F_{2}}\cap F_{1}|=|\overline{F_{1}}\cap\overline{F_{2}}|=1$. Hence, $G$ the second structure in Figure \ref{fig2}. By Lemma \ref{lem5}, $xx_{3}$ is quasi $5$-contractible, which contradicts that $G$ is contraction critically.

So $|\overline{F_{2}}\cap T_{1}|=1$.
If $|T_{1}\cap T_{2}|=2$, then $|X_{3}|=4$, and then $\overline{F_{1}}\cap\overline{F_{2}}=\{x_{3}\}$.
Let $\overline{F_{1}}\cap T_{2}=\{b\}$, $\overline{F_{2}}\cap T_{1}=\{a\}$. If $ab\in E(G)$, then there exists the first structure in Figure \ref{fig2}, a contradiction. Hence, $ab\notin E(G)$. Then we have $|\overline{F_{1}}\cap F_{2}|=0$, for otherwise $T_{1}-\{x, a\}+\{x_{3}\}$ is a nontrivial $4$-cut of $G$. Notice that, there exist the fifth structure in Figure \ref{fig2}.
We now consider that $|T_{1}\cap T_{2}|=3$, then $F_{1}\cap F_{2}\neq\emptyset$. Since $|X_{1}|=5$, $|F_{1}\cap F_{2}|=1$.
If $|F_{1}\cap\overline{F_{2}}|=0$, then there exist the fifth structure in Figure \ref{fig2}. If $|F_{1}\cap\overline{F_{2}}|=1$, then there exist the third structure ($|\overline{F_{1}}\cap F_{2}|=1$) or the forth structure ($|\overline{F_{1}}\cap F_{2}|=0$) in Figure \ref{fig2}.

By Case $1$ and Case $2$, $G$ must have some structure as shown in Figure \ref{fig2}, then we can obtain the theorem directly by Lemma \ref{lem4}, \ref{lem5}, \ref{lem6}, \ref{lem7} and \ref{lem8}.
\hfill\qedsymbol

\section{Proof of Theorem \ref{thm2}}
Let $G$ be a quasi $5$-connected graph on at least $14$ vertices. We suppose that there exists a vertex $x\in V_{4}(G)$ and $N_{G}(x)=\{x_{1}, x_{2}, x_{3}, x_{4}\}$.
In Lemma \ref{lem9} and \ref{lem10}, we consider that $G[N_{G}(x)]\cong C_{4}$ and $\{x_{1}x_{2}, x_{2}x_{3}, x_{3}x_{4}, x_{4}x_{1}\}\subseteq E(G)$, then we have $d_{G}(x_{i})\geq5$ ($i=1, 2, 3, 4$).

\begin{figure}[t]
  \centering
  \includegraphics{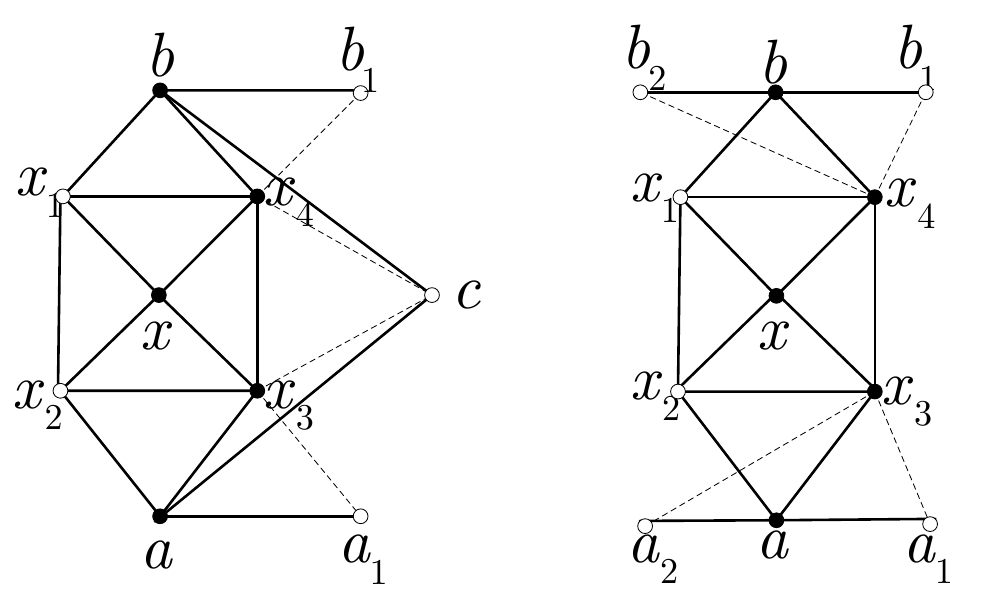}\\
  \caption{The structures of Lemma \ref{lem9} and Lemma \ref{lem10}}
  \label{fig3}
\end{figure}

\begin{lem}\label{lem9}
Let $G$ be a contraction critically quasi $5$-connected graph on at least $14$ vertices. If $G$ has the first structure in Figure \ref{fig3}, then $G$ has a quasi $5$-contractible subgraph $H$ such that $\|H\|\leq3$.
\end{lem}

\begin{proof}
If $d_{G}(x_{3})=6$ or $d_{G}(x_{4})=6$, then $G[N_{G}(a)]\cong K_{1,3}$ or $G[N_{G}(b)]\cong K_{1,3}$, and then we can get this lemma immediately by Theorem \ref{thm1}. Therefore, we assume that $d_{G}(x_{3})=d_{G}(x_{4})=5$.

{\bf Case 1.} $c\notin N_{G}(x_{3})$ and $c\notin N_{G}(x_{4})$.

We consider the graph $G^{\prime}=G/\triangle xx_{3}x_{4}$. Clearly, $\delta(G^{\prime})=4$ and $G^{\prime}$ is $4$-connected. We shall show that $G^{\prime}$ is quasi $5$-connected. Suppose that $G^{\prime}$ is not quasi $5$-connected, then $G^{\prime}$ has a nontrivial $4$-cut $T^{\prime}$. Let $F^{\prime}$ be a $T^{\prime}$-fragment and $\overline{F^{\prime}}=G^{\prime}-T^{\prime}-F^{\prime}$. Then we have $\{x_{1}, x_{2}\}\nsubseteq T$. Without loss of generality, we suppose that $F$ contains $x_{1}$ or $x_{2}$.

If $|\{x_{1}, x_{2}\}\cap T|=0$, then we have $a\in T$ and $a_{1}\in\overline{F}$. Similarly, $b\in T$ and $b_{1}\in\overline{F}$. However, we can find a nontrivial $4$-cut of $G$ wherever $c$ is, a contradiction.
So $|\{x_{1}, x_{2}\}\cap T|=1$.
Without loss of generality, we assume that $x_{1}\in T$. Same as described above, we have $a\in T$, $a_{1}\in\overline{F}$ and $\{b, b_{1}\}\cap F=\emptyset$. If $c\notin F$, then $|N_{G}(\{x, x_{3}, x_{4}, a\})\cap F|=1$, and then $|N_{G}(F-\{x_{2}\})|=3$, a contradiction. Thus, $c\in F$, and thus $b\in T$ and $b_{1}\in\overline{F}$. Then we can deduce that $|F|=|\overline{F}|=2$, and then $|V(G)|=10$, a contradiction.
This proves $\triangle xx_{3}x_{4}$ is quasi $5$-contractible.

{\bf Case 2.} $c\in N_{G}(x_{3})$ or $c\in N_{G}(x_{4})$.

Without loss of generality, we consider that $c\in N_{G}(x_{4})$.
Let $G^{\prime}$ be the graph obtained from $G$ by contracting $A:=\{b, b_{1}\}$, $B:=\{x, x_{3}\}$ to vertices $A^{\prime}$, $B^{\prime}$, respectively. Since $|V(G)|\geq12$, $\delta(G^{\prime})=4$. Then we can get $G^{\prime}$ is $4$-connected easily. Suppose that $G^{\prime}$ is not quasi $5$-connected, then $G^{\prime}$ has a nontrivial $4$-cut $T^{\prime}$ such that $|\{A^{\prime}, B^{\prime}\}\cap T^{\prime}|\neq0$. Let $F^{\prime}$ be a $T^{\prime}$-fragment and $\overline{F^{\prime}}=G^{\prime}-T^{\prime}-F^{\prime}$.

{\bf Subcase 2.1.} $|\{A^{\prime}, B^{\prime}\}\cap T^{\prime}|=2$.

Firstly, we show that, $x_{1}\in T$. Assume the contrary, $x_{1}\in F$, without loss of generality. Then $\{x_{2}, x_{4}\}\subseteq F\cup T$. Therefore, $c\in\overline{F}$, for otherwise, $T-\{b, x\}$ forms a nontrivial $4$-cut of $G$. Hence $x_{4}\in T$, it follows that $|\overline{F}|=2$ and the neighborhood set of the vertex in $\overline{F}$ that is not $c$ contains $\{b_{1}, x_{3}, c\}$. However, there are no such vertices in the graph $G$, either $c\in N_{G}(x_{3})$ or $a_{1}\in N_{G}(x_{3})$. This proves $x_{1}\in T$.
Then, we can conclude that $x_{4}\in T$ similarly. Thus, $T=\{b, b_{1}, x, x_{3}, x_{1}, x_{4}\}$, and thus $T-\{b, x\}$ forms a nontrivial $4$-cut of $G$, a contradiction.

{\bf Subcase 2.2.}  $|\{A^{\prime}, B^{\prime}\}\cap T^{\prime}|=1$.

Suppose that $A^{\prime}\in T^{\prime}$ and $B^{\prime}\in F^{\prime}$. Then we have $a_{1}x_{3}\in E(G)$, $x_{1}\in F$, $x_{4}\in T$ and $c\in\overline{F}$, for otherwise, $T-\{b\}$ forms a nontrivial $4$-cut of $G$. Since $ax_{3}\in E(G)$ and $ac\in E(G)$, $a\in T$. Since $\{x_{2}, a_{1}\}\subseteq F\cup T$, $N_{G}(\{b, x_{4}, a\})\cap\overline{F}=\{c\}$, it follows that $|N_{G}(\overline{F}-\{c\})|=3$, a contradiction.

Suppose that $A^{\prime}\in F^{\prime}$ and $B^{\prime}\in T^{\prime}$. Evidently, $x_{4}\in F$, $x_{1}\in T$ and $x_{2}\in\overline{F}$. Let $N_{G}(x_{3})\cap F:=I$. Since $ax_{2}\in E(G)$, $a\notin I$. If $I=\{x_{4}\}$, then $T-\{x, x_{3}\}+\{x_{4}\}$ forms a nontrivial $4$-cut of $G$. So $|I|=2$, which means $a_{1}\in I$ or $c\in I$. However, $N_{G}(\overline{F}-\{x_{2}\})$ forms a $3$-cut or $N_{G}(F-\{x_{4}, b, b_{1}, c\})$ forms a nontrivial $4$-cut of $G$ in both cases, a contradiction.
This proves $\{bb_{1}, xx_{3}\}$ is quasi $5$-contractible.

By Case $1$ and Case $2$, the subgraph $\triangle xx_{3}x_{4}$ or $\{bb_{1}, xx_{3}\}$ is quasi $5$-contractible.
\end{proof}

\begin{figure}[t]
  \centering
  \includegraphics{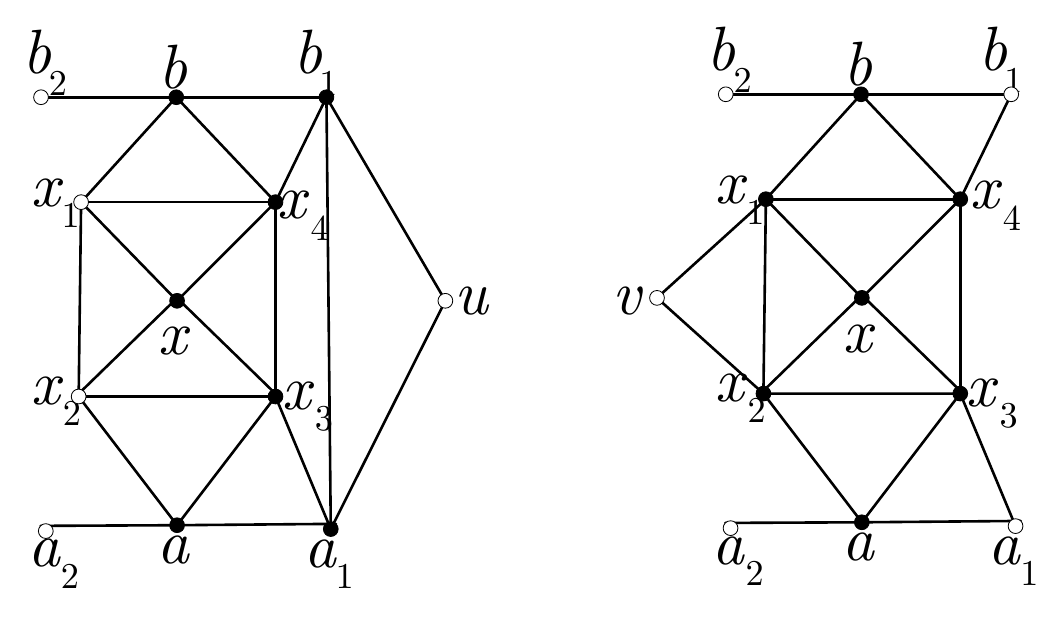}\\
  \caption{Lemma \ref{lem10}}
  \label{fig4}
\end{figure}

\begin{lem}\label{lem10}
Let $G$ be a contraction critically quasi $5$-connected graph on at least $14$ vertices. If $G$ has the second structure in Figure \ref{fig3}, then $G$ has a quasi $5$-contractible subgraph $H$ such that $\|H\|\leq3$.
\end{lem}
\begin{proof}
We just have to think about $d_{G}(x_{3})=d_{G}(x_{4})=5$.
Without loss of generality, assume that $a_{1}\in N_{G}(x_{3})$ and $b_{1}\in N_{G}(x_{4})$.

{\bf Case 1.} The graph $G$ has the first or second structure in Figure \ref{fig4}.

We firstly assume that the graph $G$ has the first structure in Figure \ref{fig4}.
Let $G^{\prime}$ be the graph obtained from $G$ by contracting $A:=\{b, b_{2}\}$, $B:=\{b_{1}, x_{4}\}$, $C:=\{x, x_{1}\}$ to vertices $A^{\prime}$, $B^{\prime}$, $C^{\prime}$, respectively. We shall show that $G^{\prime}$ is quasi $5$-connected. Clearly, $\delta(G^{\prime})=4$. And we can get $G^{\prime}$ is $4$-connected easily.
Let us assume, to the contrary, that $G^{\prime}$ has a nontrivial $4$-cut $T^{\prime}$. Let $F^{\prime}$ be a $T^{\prime}$-fragment and $\overline{F^{\prime}}=G^{\prime}-T^{\prime}-F^{\prime}$.

If $|T^{\prime}\cap\{A^{\prime}, B^{\prime}, C^{\prime}\}|>0$, then we can find a $3$-cut or nontrivial $4$-cut of $G$. We omit the specific process. This shows that $T$ is a nontrivial $4$-cut of $G$, a contradiction. So $G^{\prime}$ is quasi $5$-connected.
If $G$ has the second, we can get the same result similarly.

{\bf Case 2.} The graph $G$ has neither the first nor the second structure in Figure \ref{fig4}.

Let $G^{\prime}$ be the graph obtaining from $G$ by contracting $\triangle xx_{3}x_{4}$.
Clearly, $G^{\prime}$ is $4$-connected. We shall show that $G^{\prime}$ is quasi $5$-connected. Suppose that $G^{\prime}$ has a nontrivial $4$-cut $T^{\prime}$.
Let $F^{\prime}$ be a $T^{\prime}$-fragment and $\overline{F^{\prime}}=G^{\prime}-T^{\prime}-F^{\prime}$.
Then $|T\cap\{x_{1}, x_{2}\}|=0$.
Without loss generality, we assume that $\{x_{1}, x_{2}\}\subseteq F$. Then $\{a, b\}\subseteq T$ and $\{a_{1}, b_{1}\}\subseteq\overline{F}$. If $a_{2}\in T$, then $|F|\leq3$ and $|\overline{F}|\leq3$, and then $|V(G)|\leq12$, a contradiction. So we have $a_{2}\notin T$. Similarly, we have $b_{2}\notin T$. If $a_{2}\in F$, then $G$ has the first structure in Figure \ref{fig4}. If $a_{2}\in\overline{F}$, we can obtain that $G$ has the second structure in Figure \ref{fig4}. This proves the graph $G^{\prime}$ is quasi $5$-connected.

By Case $1$ and Case $2$, this lemma is proved.
\end{proof}

Now we are prepared to prove our Theorem \ref{thm2}.

{\bf Proof of Theorem \ref{thm2}.}
Let $N_{G}(x)=\{x_{1}, x_{2}, x_{3}, x_{4}\}$. Without loss of generality, we suppose that $\{x_{1}x_{2}, x_{2}x_{3}, x_{3}x_{4}, x_{4}x_{1}\}\subseteq E(G)$. Evidently $d_{G}(x_{i})\geq5$ and $G/xx_{i}$ is $4$-connected $(i=1,2,3,4)$. Similar to Theorem \ref{thm1}, we let $T_{1}^{\prime}$ and $T_{2}^{\prime}$ be nontrivial $4$-cuts of $G/xx_{1}$ and $G/xx_{2}$, separately. For $i\in\{1, 2\}$, let $F_{i}^{\prime}$ be a fragment of $G/xx_{i}-T_{i}^{\prime}$, $\overline{F_{i}^{\prime}}=G/xx_{i}-T_{i}^{\prime}-F_{i}^{\prime}$.
Let $F_{i}$, $T_{i}$, $\overline{F_{i}}$ be the sets in $G$ corresponding to $F_{i}^{\prime}$, $T_{i}^{\prime}$, $\overline{F_{i}^{\prime}}$ in $G/xx_{i}$. Without loss of generality, we assume that $x_{1}\in T_{1}\cap F_{2}$,  $x_{2}\in F_{1}\cap T_{2}$,  $x_{3}\in T_{1}\cap \overline{F_{2}}$ and  $x_{4}\in \overline{F_{1}}\cap T_{2}$. Furthermore, we have $1\leq|T_{1}\cap T_{2}|\leq2$, otherwise, $|V(G)|\leq11$.

If $|T_{1}\cap T_{2}|=2$, then without loss of generality, we assume that $|T_{1}\cap \overline{F_{2}}|=1$ and $|\overline{F_{1}}\cap T_{2}|=1$. Then we have $|F_{1}\cap\overline{F_{2}}|=|\overline{F_{1}}\cap F_{2}|=1$. And then the graph $G$ has the first structure in Figure \ref{fig3}.

If $|T_{1}\cap T_{2}|=1$, we have $|F_{1}\cap T_{2}|=1$ or $|\overline{F_{1}}\cap T_{2}|=1$. For otherwise $|V(G)|\leq13$, a contradiction.
Similarly, we can obtain that $|F_{2}\cap T_{1}|=1$ or $|\overline{F_{2}}\cap T_{1}|=1$. Without loss of generality, we assume that $|\overline{F_{1}}\cap T_{2}|=1$ and $|\overline{F_{2}}\cap T_{1}|=1$. Thus, the graph $G$ has the second structure in Figure \ref{fig3}.

By Lemma \ref{lem9} and Lemma \ref{lem10}, we know that $G$ has a quasi $5$-contractible subgraph $H$ such that $\|H\|\leq3$.
\hfill\qedsymbol

\end{document}